\documentclass{article}

\usepackage{amssymb,amsmath,amsthm,amsfonts}
\usepackage{xcolor}
\usepackage{graphicx}

\newtheorem{theorem}{Theorem}[section]

\newtheorem{lemma}[theorem]{Lemma}
\newtheorem{proposition}[theorem]{Proposition}
\newtheorem{remark}[theorem]{Remark}
\newtheorem{example}[theorem]{Example}
\newtheorem{assumption}[theorem]{Assumption}

\newcommand{\supp}{\textnormal{supp}\,}
\newcommand{\sinc}{\textnormal{sinc}}

\title{A sampling theorem on shift-invariant spaces associated with the fractional Fourier transform domain}
\author{Sinuk Kang}

\date{}

\begin{document}
\maketitle

\begin{abstract}
As a generalization of the Fourier transform, the fractional Fourier transform was introduced and has been further investigated both in theory and in applications of signal processing. 
We obtain a sampling theorem on shift-invariant spaces associated with the fractional Fourier transform domain. The resulting sampling theorem extends not only the classical Whittaker-Shannon-Kotelnikov sampling theorem associated with the fractional Fourier transform domain, but also extends the prior sampling theorems on shift-invariant spaces.
\newline

\newline
Mathematics Subject Classification (2010) Primary  94A20 $\cdot$ Secondary 42C15
\end{abstract}

\section{Introduction}
Let $PW_\pi := \{ f \in L^2(\mathbb R) \cap C(\mathbb R) : \supp \hat{f} \subset [-\pi, \pi] \}$ be the Paley-Wiener space of signals band-limited to $[-\pi, \pi]$ where we take the Fourier transfom (FT) as $  \mathcal{F}[f](\xi) =\hat{f}(\xi) := \frac{1}{\sqrt{2\pi}}\int_{\mathbb R} f(t) e^{-it\xi} dt$ for $f(t) \in L^1(\mathbb R) \cap L^2 (\mathbb R )$. Then the Whittaker-Shannon-Kotelnikov (WSK) sampling theorem \cite{Kotelnikov:1933vx,Shannon:1949uo,Whittaker:1915wb} says that  any $f(t)$ in $ PW_\pi $ is determined by its samples $\{ f(n) : n \in \mathbb Z \}$, and can be reconstructed via 
\begin{equation*}	
	f(t) = \sum_{n \in \mathbb Z} f(n) \textnormal{sinc}(t-n)
\end{equation*} where $\textnormal{sinc} (t):=\frac{\sin \pi t}{\pi t} $. 

As a generalization of FT, the fractional Fourier transform (FrFT) was introduced in 1980s \cite{McBRIDE:1987el,NAMIAS:1980hv}. 
For any $f \in L^2 (\mathbb R)$ and $\theta \in \mathbb R$ we let  
	\begin{equation}\label{defFrFT}
		\mathcal{F}_{\theta}[f](\xi) = \hat{f}_{\theta} (\xi) := \int_{\mathbb R} f(t) {K}_{\theta}(t, \xi) dt
	\end{equation} be the FrFT of $f(t)$ with respect to $\theta$	
	where 
	\begin{equation*}
		{K}_{\theta}(t,\xi) = 
		\left\{
			\begin{array}{ll}
				\delta(t-\xi) & \textnormal{if }  \theta = 2\pi n,~n \in \mathbb Z  \\
				\delta(t+\xi) & \textnormal{if }  \theta+\pi= 2\pi n,~n \in \mathbb Z \\
				c(\theta) e^{ia(\theta)(t^2+\xi^2) -ib(\theta)t \xi} & \textnormal{otherwise}
			\end{array}
		\right.
	\end{equation*}
is the transformation kernel with $c(\theta) = \sqrt{\frac{1 - i \cot \theta}{2\pi}}$, $a(\theta) = \frac{\cot \theta}{2}$, and $b(\theta) = \csc\theta$. 
The inverse FrFT with respect to $\theta$ is defined by the FrFT with respect to $-\theta$, that is, 
	\begin{equation*}
		f(t) = \int_{\mathbb R} \hat{f}_{\theta} (\xi) {K}_{-\theta} (t,\xi) d\xi.
	\end{equation*} 
For notational ease we write $a(\theta)$ and $b(\theta)$ as $a$ and $b$, respectively, whereas we keep $c(\theta)$ to avoid a confusion over $c(n)$ of a sequence $\mathbf{c} = \{ c(n) \}_n $. 
Note that the FrFT is a unitary operator from $L^2(\mathbb R)$ onto $L^2(\mathbb R)$ and corresponds to the FT when $\theta = \frac{\pi}{2}$. 
Let 
\begin{equation*}
	\mathcal{M}_\theta [f] (t) = \vec{f}(t) := \lambda_\theta (t) f(t)
\end{equation*} 
be a unitary operator from $L^2(\mathbb R)$ onto $L^2(\mathbb R)$
where 
\begin{equation} \label{v0.7eq1}
	\lambda_{\theta} (\cdot) := e^{ia(\cdot)^2}
\end{equation} denotes a domain independent modulation function \cite{Zayed:2011hd}.
We call $\vec{f}(t)$ the chirp-modulated version of $f(t)$.
Then it follows, by definition, that 
\begin{equation} \label{v1eq2.01}
	\hat{f}_\theta (\xi) = c(\theta)\mathcal{M}_\theta [\hat{\vec{f}}(b\xi)](\xi).
\end{equation} 
\eqref{v1eq2.01} implies that $f(t)$ is band-limited to $[-\omega, \omega]$ in the FrFT domain if and only if $\vec{f}(t)$ is band-limited to $[-\frac{\omega}{|b|},\frac{\omega}{|b|}]$ in the FT domain. Based on this observation, one can easily derive a sampling expansion of band-limited signals in the FrFT domain from the corresponding one of band-limited signals in the FT domain \cite{Candan:2003ih,Tao:2008dg,Torres:2006ja,Zayed:1996hg}. For instance, assume that $f(t)$ is band-limited to $[-\pi, \pi]$ in the FrFT domain. Since $ \supp\hat{f}_\theta(\xi) \subseteq [-\pi, \pi] $ is equivalent to $\supp\hat{\vec{f}} (\xi) \subseteq [-\frac{\pi}{|b|},\frac{\pi}{|b|}]$, one can deduce from the well-known sampling expansion of $\vec{f}$, $\vec{f}(t) =  \sum_{n \in \mathbb Z } \vec{f}({bn})\, {\sinc} \frac{1}{b}(t-n) $, that $${f}(t) = e^{-iat^2} \sum_{n \in \mathbb Z } {f}({bn})e^{-ia(bn)^2} {\sinc} \frac{1}{b}(t-n).$$

On the other hand, Hardy \cite{Hardy:1941wq} pointed out that the WSK sampling theorem can be proved by so-called Fourier duality between $PW_\pi$ and $L^2 [0,2\pi]$. To describe it we first define the Zak transform. For any $\phi(t) \in L^2(\mathbb R)$ let $Z_\phi (t,\xi) := \sum_{n \in \mathbb Z} \phi(t-n) e^{in\xi}$ be the Zak transform \cite{JANSSEN:1988ul} of $\phi(t)$. Then $Z_\phi(t, \xi)$ is well-defined a.e. on $\mathbb R^2$ and quasi-periodic in the sense that $Z_\phi(t+1, \xi) = Z_\phi (t,\xi) e^{i\xi}$ and $Z_\phi(t,\xi+2\pi) = Z_\phi (t,\xi)$.  

Now consider an integral operator $\mathcal{J}_{PW}$ from $L^2 [0,2\pi]$ into $PW_\pi$ with the kernel $Z_{\sinc}(t,\xi)$, defined by
\begin{equation*}
	\mathcal{J}_{PW}F (t) :=\langle F(\cdot), \overline {Z_{\sinc}(t,\cdot)} \rangle =  \sum_{n \in \mathbb Z} \langle F(\cdot),  e^{-in\cdot}  \rangle \sinc(t-n),~ F \in L^2 [0,2\pi].
\end{equation*}
Here $\langle \cdot , \cdot \rangle$ denotes the inner product in $L^2[0,2\pi]$. Then $\mathcal{J}_{PW}$ is a bounded invertible operator from $L^2 [0,2\pi]$ onto $PW_\pi$. Thus we obtain for any $F(\xi) \in L^2[0,2\pi]$
\begin{equation*}
	f(t) : = \mathcal{J}_{PW}F (t)   = \sum_{n \in \mathbb Z} f(n) \sinc(t-n)
\end{equation*} which is the WSK sampling expansion on $PW_\pi$.
Motivated by Hardy's approach, Garcia et al. \cite{Garcia:2005fz} define an integral operator $\mathcal{J}$ from $L^2[0,2\pi]$ onto its range, say $V(\phi)$, by
\begin{equation*}
	\mathcal{J}F (t) :=\langle F(\cdot), \overline {Z_\phi(t,\cdot)} \rangle,~ F \in L^2 [0,2\pi]
\end{equation*}for some $\phi(t) \in L^2(\mathbb R)$.
They prove that under a suitable condition on $\phi(t)$, $\mathcal{J}$ is a bounded invertible operator from $L^2[0,2\pi]$ onto $V(\phi)$ and $\{ Z_\phi(n,\xi) : n \in \mathbb Z\}$ forms a stable basis (or a frame) of $V(\phi)$. Thus, by the same token, it follows via $\mathcal{J}$ that
\begin{equation*}
	f(t) = \sum_{n \in \mathbb Z } f(n) S_n (t),~ f(t) \in V(\phi)
\end{equation*}where $S_n (t)$ is in $V(\phi)$ such that $\{ \mathcal{J}[S_n](\xi) : n \in \mathbb Z\}$ is the dual of $\{ Z_\phi(n,\xi) : n \in \mathbb Z\}$ in $L^2[0,2\pi]$. Note that $V(\phi)=  \overline{ \textnormal{span}} \{ \phi(t-n) : n \in \mathbb Z\}$ and $f(t) \in V(\phi)$ implies $f(t-n) \in V(\phi)$ for $n \in \mathbb Z$. We call $V(\phi)$ the shift-invariant space generated by $\phi(t)$. Obviously, this approach reduces to Hardy's when $\phi(t) = \sinc(t)$.

$V(\phi)$, as a generalization of $PW_\pi$(=V(\sinc)), has been widely studied from the beginning of 1990s. The structure of $V(\phi)$ is addressed by so-called fiber analysis in \cite{Bownik:2000uy,DEBOOR:1994ed,RON:1995wk}. A sampling theorem on $V(\phi)$ has been investigated in various directions. See \cite{Aldroubi:2002cz,Chen:2002im,Garcia:2006ir,Kang:2011bf,Kim:2008ik,Sun:2003wq,Walter:2005do} and references therein.

Recently, Bhandari and Zayed \cite{Zayed:2011hd} extend $V(\phi)$ into a chirp-modulated shift-invariant space
\begin{equation*}
	V_\theta (\phi)  := {closure}\, \{  (\mathbf{c} \ast_\theta \phi )(t) : \mathbf{c} \in \ell^2(\mathbb Z) \}.
\end{equation*} 
Here $\ast_\theta$ is the fractional semi-discrete convolution to be specified later in \eqref{v0.2eq1.01}. 
$V_\theta(\phi)$ corresponds to $V(\phi)$ when $\theta=\frac{\pi}{2}$, and is no longer (integer) shift-invariant.
Using \eqref{v1eq2.01}, they obtain a sampling expansion on $V_\theta(\phi)$. 
However there is a gap in determining the structure of $V_\theta(\phi)$ (Remark \ref{v1.4gap}) and their sampling expansion is formally presented: no condition is given for such sampling expansion to hold in $V_\theta(\phi)$.

In this paper, motivated by \cite{Garcia:2005fz}, we define a bounded invertible operator $\mathcal{J}_\theta$ from $L^2[0,\frac{2\pi}{|b|}]$ onto $V_\theta(\phi)$, and use it to derive a sampling expansion on $V_\theta(\phi)$. We call this isomorphic relationship the fractional Fourier duality. The structure of $V_\theta(\phi)$ is also addressed.
Our results extend those of \cite{Garcia:2005fz,Kim:2008ik} and improve the previous sampling theory on $V_\theta(\phi)$ by filling the gap in \cite{Zayed:2011hd}.

This paper is organized as follows. In Section \ref{preliminary} we define the notation and terminology needed throughout the paper and provide useful properties of them. In Section \ref{secRKHS} a condition for $V_\theta(\phi)$ to be a reproducing kernel Hilbert space is presented. In Section \ref{FrFourierDuality} the fractional Fourier duality is defined. We discuss conditions under which a certain sequence of functions forms a frame or a Riesz basis of $L^2[0,\frac{2\pi}{|b|}]$. Using this together with the Fourier duality, we derive a sampling expansion on $V_\theta(\phi)$. As an illustrating example the sampling theorem of band-limited signals in the FrFT domain is given.

\section{Preliminary}\label{preliminary}

A sequence $\{\phi_n:n\in \mathbb{Z}\}$ of vectors in a separable Hilbert space $\mathcal{H}$ is

\begin{itemize}

\item a Bessel sequence with a bound $B$ if there is constant $B>0$ such that
\begin{equation*}
 \sum\limits_{n\in \mathbb{Z}}|\langle
\phi ,\phi _{n}\rangle|^{2}\leq B\Vert \phi \Vert ^{2},~\phi \in
\mathcal{H};
\end{equation*}
\item  a frame of $\mathcal{H}$ with bounds $(A,B)$ if there are constants $B \geq A >0$ such that
\begin{equation*}
A\Vert \phi \Vert ^{2}\leq \sum\limits_{n\in \mathbb{Z}}|\langle
\phi ,\phi _{n}\rangle|^{2}\leq B\Vert \phi \Vert ^{2},~\phi \in
\mathcal{H};
\end{equation*}

\item  a Riesz basis of $\mathcal{H}$ with bounds $(A,B)$ if it is complete in $\mathcal{H}$ and there are constants $B \geq A >0$ such that
\begin{equation*}
A\Vert \mathbf{c}\Vert ^{2}\leq \Vert \sum\limits_{n\in
\mathbb{Z}}c(n)\phi _{n}\Vert ^{2}\leq B\Vert \mathbf{c}\Vert
^{2},~\mathbf{c}=\{c(n)\}_{n\in \mathbb{Z}}\in \ell^{2},
\end{equation*}
where $\Vert \mathbf{c}\Vert ^{2}:=\sum\limits_{n\in
\mathbb{Z}}|c(n)|^{2}.$
\end{itemize}

We introduce some useful properties of the FrFT \cite{Almeida:1994hr,Zayed:1998bg}:
\begin{itemize}
	\item $\langle f,g \rangle_{L^2(\mathbb R)} = \langle \hat{f}_\theta, \hat{g}_\theta \rangle_{L^2 (\mathbb R)}$ (Parseval's relation);
	\item $ \mathcal{F}_\theta [ \frac{c(\theta)}{\sqrt{2\pi}} \overline{\lambda_\theta (t)} (\vec{f} \ast \vec{g})(t)](\xi)  =  \hat{f}_\theta(\xi) \hat{g}_\theta (\xi) \overline{\lambda_\theta (t)} $.
\end{itemize} 
We have defined the FrFT $\mathcal{F}_\theta$ in the previous section as a generalization of the FT $\mathcal{F}$. Analogously, we define useful notations associated with $\theta$, most of which are borrowed from \cite{Zayed:2011hd}.
For any $\phi(t) \in L^2(\mathbb R)$ let 
	\begin{equation} \label{v0.8eq1.1}
		C_{\phi, \theta}(t) := \sum_{n \in \mathbb Z} |\vec{\phi} (t-n)|^2 \textnormal{ and } G_{\phi, \theta}(\xi) := \sum_{n \in \mathbb Z} | \hat{\phi}_{\theta} (\xi + n\Delta)|^2
	\end{equation} where 
	\begin{equation}
		\Delta := \frac{2\pi}{|b|}= 2\pi|\sin\theta|.
	\end{equation}
Note that $\| \phi(t) \|^2_{L^2(\mathbb R)}  = \| C_{\phi, \theta} (t)\|^2_{L^1[0,1]} = \| G_{\phi,\theta}(\xi) \|^2_{L^1[0,\Delta]}$.
For any $\mathbf{c} = \{ c(n)  \}_n \in \ell^2 (\mathbb Z)$ let  
	\begin{equation*}
		\hat{\mathbf{c}}_{\theta} (\xi) = \sum_{n \in \mathbb Z} c(n) {K}_{\theta} (n, \xi)
	\end{equation*}
be the discrete FrFT of $\mathbf{c}$ with respect to $\theta$.  	
The inverse discrete FrFT is $c(n) = \int_0^{\Delta} \hat{\mathbf{c}}_{\theta}(\xi) {K}_{-\theta} (n, \xi) d\xi$. We remark that, as a consequence of Lemma \ref{v0.3lem3.3}, $\{ {K}_{\theta} (n, \xi) : n \in \mathbb R\}$ is an orthonormal basis of $L^2[0,\Delta]$.

A Hilbert space $\mathcal{H}$ consisting of complex valued functions on a set
$E$ is called a reproducing kernel Hilbert space(RKHS) if
there is a function
$q(s,t)$ on $E\times E$, called the reproducing kernel of $\mathcal{H}$,
satisfying
\begin{itemize}
\item $q(\cdot,t)\in \mathcal{H}$ for each $t$ in $E$,
\item $\langle f(s), q(s,t)\rangle=f(t)$, $f\in \mathcal{H}$.
\end{itemize}
In an RKHS $\mathcal{H}$ any norm converging sequence also converges uniformly on any subset of $E$, on which $\|q(\cdot, t)\|_{\mathcal{H}}^{2} = q(t,t)$ is bounded. The reproducing kernel of $\mathcal{H}$ is unique and obtained by $q(s,t) = \sum_{n \in \mathbb Z} e_n (s) \overline{e_n (t)}$ where $\{ e_n : n \in \mathbb Z \}$ is an orthonormal basis of $\mathcal{H}$ \cite{Higgins:1996uo}. 

For any $\mathbf{c} = \{ c(n) \}_n \in \ell^2 (\mathbb Z)$ and $ \phi (t) \in L^2(\mathbb R) $ let 
\begin{equation} \label{v0.2eq1.01}
	T_\theta (\mathbf{c}) = ( \mathbf{c} \ast_{\theta} \phi )(t):=  c(\theta)\overline{\lambda_\theta(t)} (\vec{\mathbf{c}} \ast \vec{\phi})(t)
\end{equation}
be the fractional semi-discrete convolution of $\mathbf{c}$ and $\phi(t)$ with respect to $\theta$. Here $\ast$ denotes the conventional semi-discrete convolution, i.e. $(\mathbf{c} \ast \phi)(t) = \sum_{n \in \mathbb Z} c(n) \phi(t-n)$. It should be noticed that \eqref{v0.2eq1.01} may or may not converge in $L^2 (\mathbb R)$.

\begin{lemma} [cf. Lemma 1 of \cite{Zayed:2011hd}]\label{v0.8lemma4.1}
Let $\phi(t) \in L^2 (\mathbb R)$. If $\{ \vec{\phi} (t-n) : n \in \mathbb Z \}$ is a Bessel sequence then for any $\mathbf{c} = \{ c(n)\}_n \in \ell^2(\mathbb Z)$, $( \mathbf{c} \ast_{\theta} \phi )(t) $ converges in $L^2(\mathbb R)$ and
	\begin{equation*}
		\mathcal{F}_{\theta}[\mathbf{c} \ast_{\theta} \phi](\xi) = \overline{\lambda_{\theta} (\xi)} \hat{\mathbf{c}}_{\theta} (\xi)  \hat{\phi}_\theta (\xi).
	\end{equation*}
\end{lemma}

\begin{proof}
We refer to Lemma 7.1.6 of \cite{Christensen:2008wy} and Lemma 1 of \cite{Zayed:2011hd}. 
\end{proof}

\begin{remark}
To make sure $(c\ast_\theta \phi)(t)$ to be well-defined in $L^2(\mathbb R)$, we assume in Lemma \ref{v0.8lemma4.1} that $\{ \vec{\phi} (t-n) : n \in \mathbb Z \}$ is a Bessel sequence, while the authors of \cite{Zayed:2011hd} assumed $\phi$ to have a compact support (Lemma 1 of \cite{Zayed:2011hd}).
\end{remark}
For $\phi(t) \in L^2(\mathbb R)$ we have defined
\begin{equation*}
	V_\theta (\phi)  := {closure}\, \{  (\mathbf{c} \ast_\theta \phi )(t) : \mathbf{c} \in \ell^2(\mathbb Z) \}
\end{equation*} 
as the chirp-modulated shift-invariant space generated by $\phi(t)$. As already mentioned $V_\theta(\phi)$ may or may not be well-defined in $L^2 (\mathbb R)$.
For notational ease we let
\begin{equation} \label{v0.4eq1.1}
	\eta_{\theta}(t) := c(\theta) \overline{\lambda_\theta(t)}
\end{equation} 
so that
\begin{equation*}
	V_\theta(\phi)  = \overline{\textnormal{span}} \{ \eta_{\theta} (t)\vec{\phi}(t-n) : n \in \mathbb Z \}.
\end{equation*}
Note that $\sqrt{\Delta}|\eta_{\theta}(t)| =\sqrt{\Delta} |c(\theta)|  = 1$ for $t \in \mathbb R$.
For any $\phi(t) \in L^2 (\mathbb R)$ let
\begin{equation} \label{v0.9defZAK} 
	Z_{\phi, \theta} (t, \xi) := \sum_{n \in \mathbb Z} \vec{\phi} (t-n) \lambda_\theta(n) \overline{{K}_\theta (n, \xi)} 
\end{equation} 
be the fractional Zak transform of $\phi(t)$ with respect to $\theta$. As a consequence of Lemma \ref{v0.3lem3.3}, $\{\lambda_\theta(n) {K}_\theta (n , \xi) : n \in \mathbb Z \}$ is an orthononal basis of $L^2[0,\Delta]$ so that $  \| Z_{\phi, \theta} (t, \cdot) \|^2_{L^2[0,\Delta]} = \sum_{n \in \mathbb Z}|\vec{\phi} (t-n) |^2$ a.e. on $[0,1]$. By Lebesgue's dominated convergence theorem we have
\begin{eqnarray*}
	\| Z_{\phi, \theta} (t,\xi)\|^2_{L^2([0,1]\times[0,\Delta])}  = \int_0^1 \sum_{n \in \mathbb Z} | \vec{\phi} (t-n)|^2 dt 
	 =  \sum_{n \in \mathbb Z} \int_0^1 | {\phi} (t-n)|^2 dt = \| \phi \|^2_{L^2 (\mathbb R)}
\end{eqnarray*}
and, therefore, $Z_{\phi, \theta} (t, \xi)$ is well-defined a.e. on $[0,1]\times[0,\Delta]$. Since 
	\begin{equation}	 \label{v0.2eq2}
		Z_{\phi, \theta} (t-n, \xi) =  Z_{\phi, \theta} (t, \xi) e^{-ibn\xi},~n \in \mathbb Z 
	\end{equation} and	
$Z_{\phi, \theta} (t, \xi +\Delta) = Z_{\phi, \theta} (t, \xi) e^{-ia\Delta(\Delta + 2\xi)}$, $Z_{\phi, \theta} (t, \xi)$ is well-defined a.e. on $\mathbb R^2$. 
When $\theta= \frac{\pi}{2}$, $Z_{\phi,\theta}(t,\xi)$ corresponds to $Z_\phi (t,\xi)$ where $Z_\phi (t,\xi):=\sum_{n \in \mathbb Z} \phi(t-n) e^{in\xi}$ denotes the conventional Zak transform \cite{JANSSEN:1988ul} of $\phi(t)$. 

\begin{remark}
Bhandari and Zayed define the fractional Zak transform in different way (cf. Definition 9 of \cite{Zayed:2011hd}). Almost all properties of their fractional Zak transform coincide with ours except \eqref{v0.2eq2}.
\end{remark}

We close this section with the following useful observations:
\begin{eqnarray} 
	&&   C_{\phi,\theta}(t) = C_\phi (t); \label{v1eq2.02} \\
	&&   G_{\phi,\theta}(\xi) =\frac{1}{\sqrt\Delta} G_{\vec\phi} (b\xi);  \label{v1eq2.001}\\
	&&   Z_{\phi,\theta}(t,\xi) =  \overline{\lambda_\theta (\xi)} Z_{\vec\phi} (t,b\xi), \label{v1eq2.03}
\end{eqnarray} where $C_\phi (t)$, $G_\phi(\xi)$ and $Z_\phi(t,\xi) $ denote $C_{\phi,\frac{\pi}{2}}(t)$, $G_{\phi,\frac{\pi}{2}}(\xi)$ and $Z_{\phi,\frac{\pi}{2}}(t,\xi)$, respectively.

\section{$V_\theta({\phi})$ as an RKHS} \label{secRKHS}

The aim of Section \ref{secRKHS} is to determine $V_\theta(\phi)$ as an RKHS.
In \cite{Kim:2008ik}, Kim and Kwon show that $V(\phi)$ becomes an RKHS depending on behavior of $C_{\phi} (t)$. Their result, however, can not be directly applied to $V_\theta(\phi)$ since it is no longer a shift-invariant space.

Consider the linear operator $\mathcal{U}_\theta:V(\vec{\phi}) \rightarrow V_\theta(\phi)$ defined by
\begin{equation*}
	(\mathcal{U}_\theta f )(t) := \eta_{\theta}(t)f(t) ,~f  \in V(\vec{\phi}).
\end{equation*} 
$\mathcal{U}_\theta$ is an isomorphism from $V(\vec{\phi})$ onto $V_\theta(\phi)$ and $\| \mathcal{U}_\theta f \|^2_{L^2(\mathbb R)} = \frac{1}{\Delta} \| f \|^2_{L^2(\mathbb R)} $. It is easy to see that for $V_\theta(\phi)$ to be an RKHS, it is necessary that $\phi(t)$ is well-defined everywhere on $\mathbb R$ and $C_{\phi, \theta} (t) < \infty$ for $t \in \mathbb R $.

Conversely, following Proposition 2.3. of \cite{Kim:2008ik}, we have:
\begin{proposition} \label{v1prop3.2}
Let $\phi(t) \in L^2 (\mathbb R)$ be well-defined everywhere on $\mathbb R$, and let $C_{\phi, \theta}(t)$ and $\eta_\theta (t)$ be given by \eqref{v0.8eq1.1} and \eqref{v0.4eq1.1}, respectively. Assume $C_{\phi, \theta}(t) < \infty$ for $t \in \mathbb R$. 
\begin{enumerate}
	\item[(a)] If $\{\eta_\theta(t) \vec{\phi} (t-n) : n \in \mathbb Z \}$ is a frame of $V_{\theta}(\phi)$, then $V^p_\theta (\phi):=\{ (\mathbf{c} \ast_\theta \phi)(t) : \mathbf{c} \in N(T_\theta)^{\perp} \}$ is an RKHS in which any $f(t) = (\mathbf{c} \ast_\theta \phi )(t)$ is the pointwise limit of $\eta_\theta(t)\sum_{n \in \mathbb Z} \vec{c}(n) \vec{\phi} (t-n) $, which converges also in $L^2(\mathbb R)$. Here $N(T_\theta)^{\perp}$ denotes the orthogonal complement of $N(T_\theta):=Ker(T_\theta)$ in $\ell^2(\mathbb Z)$
	
	\item[(b)] If $\{\eta_\theta(t) \vec{\phi} (t-n) : n \in \mathbb Z \}$ is a Riesz basis of $V_{\theta}(\phi)$, then $V_\theta (\phi)$ is an RKHS in which any $f(t) = (\mathbf{c} \ast_\theta \phi )(t)$ is the pointwise limit of $\eta_\theta(t)\sum_{n \in \mathbb Z} \vec{c}(n) \vec{\phi} (t-n) $, which converges also in $L^2(\mathbb R)$.

	\item[(c)] If $\{\eta_\theta(t) \vec{\phi} (t-n) : n \in \mathbb Z \}$ is a frame of $V_{\theta}(\phi)$, $\vec{\phi}(t)$ is continuous on $\mathbb R$ and $\sup_{\mathbb R} C_{\phi, \theta}(t) < \infty$, then then $V_\theta (\phi)$ is an RKHS in which any $f(t) = (\mathbf{c} \ast_\theta \phi )(t)$ is the pointwise limit of $\eta_\theta(t) \sum_{n \in \mathbb Z} \vec{c}(n) \vec{\phi} (t-n) $, which converges also in $L^2(\mathbb R)$ and uniformly on $\mathbb R$ to a continuous function on $\mathbb R$ (so $V_{\theta}(\phi) \subset C(\mathbb R) \cup L^2 (\mathbb R)$).
\end{enumerate}
\end{proposition}

\begin{proof}
Thanks to $\mathcal{U}_\theta$, $\{  \eta_{\theta} (t)\vec{\phi}(t-n) : n \in \mathbb Z \}$ is a Bessel sequence, a frame, a Riesz basis, or an orthonormal basis of $V_{\theta}(\phi)$ if and only if $\{ \vec{\phi} (t-n) : n \in \mathbb Z \}$ is a Bessel sequence, a frame, a Riesz basis, or an orthonormal basis of $V(\vec{\phi})$, respectively. We refer to the proof of Proposition 2.3. of \cite{Kim:2008ik}. Then it is sufficient to show that (i) the point evaluation operator $l_t (\cdot): V_\theta(\phi) \rightarrow \mathbb R$, defined by $l_t (f) := f(t)$ for $t \in \mathbb R$, is bounded, and (ii) $Ker(T_\theta) = Ker(T)$ where $T:= \sqrt{2\pi} T_{\frac{\pi}{2}}$. (i) follows since $ l_t (f)  = l_t (\mathcal{U}^{-1}_\theta \tilde{f}) = \overline{c(\theta)}{\lambda_\theta(t)} l_t (\tilde{f})$ where $\tilde{f} := \mathcal{U}_\theta f $. 
Since $T_\theta(\mathbf{c}) = \eta_\theta(t) T(\vec{\mathbf{c}})$ for $\mathbf{c} \in \ell^2$, we have (ii).
\end{proof}

We have actually proved that $V(\vec\phi)$ is an RKHS if and only if $V_\theta(\phi)$, as an isomorphic image of $V(\vec\phi)$ through $\mathcal{U}_\theta$, is an RKHS. This is not true with an arbitrary isomorphic image of $V(\vec\phi)$. Note here that $q_\theta (s,t) = \eta_\theta(s) \overline{\eta_\theta (t)}q(s,t)$ where $q_\theta(s,t)$ and $q(s,t)$ denote the reproducing kernels of $V_\theta(\phi)$ and $V(\vec\phi)$, respectively. 

We are now interested in the following question: When does $\{\eta_\theta(t) \vec{\phi} (t-n) : n \in \mathbb Z \}$ form a frame or a Riesz basis of $V_\theta(\phi)$? 
If $\theta = \frac{\pi}{2}$ the behavior of a sequence of the form $\{ \eta_\theta(t) \vec{\phi} (t-n) : n \in \mathbb Z \}$ in $V_\theta(\phi)$, can be characterized by so-called fiber analysis. We refer to \cite{Bownik:2000uy,Christensen:2008wy,DEBOOR:1994ed,Kim:2008ik} for details. 
We adapt Theorem 7.1.7 of \cite{Christensen:2008wy} and translate it into the following two propositions. The proofs are omitted since they are essentially the same as the proof Theorem 7.1.7 of \cite{Christensen:2008wy}, provided that one notices $\| T_\theta (\mathbf{c}) \|^2_{L^2 (\mathbb R)} = \| \mathcal{F}_\theta [ T_\theta (\mathbf{c}) ] \|^2_{L^2 (\mathbb R)}  = \| \overline{\lambda_\theta(\xi)} \hat{\mathbf{c}}_\theta(\xi) \hat{\phi}_\theta(\xi) \|^2_{L^2 (\mathbb R)} = \int^{\Delta}_{0} | \hat{\mathbf{c}}_\theta(\xi) |^2 G_{\phi, \theta} (\xi) d\xi$.

\begin{proposition} \label{v0.2prop2.2}Let $\phi(t) \in L^2 (\mathbb R)$, $B \geq A > 0$ and let $\eta_\theta(t)$ be given by \eqref{v0.4eq1.1}. Then $\{ \eta_\theta(t) \vec{\phi} (t-n) : n \in \mathbb Z \}$ is
	\begin{enumerate}
		\item[(a)] a Bessel sequence with bound $B$ if and only if $G_{\phi,\theta}(\xi) \leq B ~\textnormal{a.e. on }[0,\Delta]$; 

		\item[(b)] a frame of $V_{\theta}(\phi)$ with bound $(A,B)$ if and only if 
			\begin{equation*}
				 A \leq G_{\phi,\theta}(\xi) \leq {B}~\textnormal{a.e. on } [0,\Delta] \cap \textnormal{ supp}\, G_{\phi,\theta}(\xi);
			\end{equation*}

		\item[(c)] a Riesz basis of $V_{\theta}(\phi)$ with bound $(A,B)$ if and only if 
			\begin{equation*}
				 A \leq G_{\phi,\theta}(\xi) \leq B~\textnormal{a.e. on }[0,\Delta]; 
			\end{equation*}
			
		\item[(d)] an orthonormal basis of $V_{\theta}(\phi)$ if and only if $G_{\phi,\theta}(\xi) = 1 ~\textnormal{a.e. on }[0,\Delta]$.
	\end{enumerate}
\end{proposition}

Or equivalently, we have:

\begin{proposition} \label{v0.3prop2.4} Let $\phi(t) \in L^2 (\mathbb R)$ and $B \geq A > 0$, and let $T_\theta$ and $\eta_\theta(t)$ be given by \eqref{v0.2eq1.01} and \eqref{v0.4eq1.1}, respectively. Then $\{ \eta_\theta(t) \vec{\phi} (t-n) : n \in \mathbb Z \}$ is
	\begin{enumerate}
		\item[(a)]  a Bessel sequence with bound $B$ if and only if $T_\theta$ is a bounded linear operator from $\ell^2(\mathbb Z)$ into $V_\theta(\phi)$ and $ \| T_\theta (\mathbf{c}) \|^2_{L^2(\mathbb R)}  \leq B \| \mathbf{c}\|^2,~ \mathbf{c}  \in \ell^2(\mathbb Z)$; 

		\item[(b)] a frame of $V_{\theta}(\phi)$ with bound $(A,B)$ if and only if $T_\theta$ is a bounded linear operator from $\ell^2(\mathbb Z)$ onto $V_\theta(\phi)$ and
			\begin{equation*}
				 A  \| \mathbf{c}\|^2 \leq \| T_\theta (\mathbf{c}) \|^2_{L^2(\mathbb R)}  \leq B \| \mathbf{c}\|^2,~ \mathbf{c}  \in {N}(T_\theta)^{\perp} 
			\end{equation*} where ${N}(T_\theta)^{\perp}$ is the orthogonal complement of ${N}(T_\theta) := {Ker}(T_\theta)$.

		\item[(c)] a Riesz basis of $V_{\theta}(\phi)$ with bound $(A,B)$ if and only if $T_\theta$ is an isomorphism from $\ell^2(\mathbb Z)$ onto $V_\theta(\phi)$ and 
			\begin{equation*}
				  A  \| \mathbf{c}\|^2 \leq \| T_\theta (\mathbf{c}) \|^2_{L^2(\mathbb R)}  \leq B \| \mathbf{c}\|^2,~ \mathbf{c}  \in \ell^2(\mathbb Z);
			\end{equation*}
			
		\item[(d)] an orthonormal basis of $V_{\theta}(\phi)$ if and only if $T_\theta$ is a unitary operator from $\ell^2(\mathbb Z)$ onto $V_\theta(\phi)$ and $ \| T_\theta (\mathbf{c}) \|^2_{L^2(\mathbb R)} = \| \mathbf{c}\|^2,~ \mathbf{c}  \in \ell^2(\mathbb Z)$.
	\end{enumerate}
\end{proposition}

\begin{remark} \label{v1.4gap}
In \cite{Zayed:2011hd}, the authors claim that the formula (26) of \cite{Zayed:2011hd} is the necessary and sufficient condition for $\{ \vec{\phi}(t-n) : n \in \mathbb Z \}$ to be a Riesz basis of $V_\theta(\phi)$. However it is the one for $\{ \vec{\phi}(t-n) : n \in \mathbb Z \}$ to be a Riesz basis of $V(\vec{\phi})$, but not of $V_\theta(\phi)$. Or equivalently, it is the necessary and sufficient condition for $\{ \eta_{\theta}(t)\vec{\phi}(t-n) : n \in \mathbb Z \}$ to be a Riesz basis of $V_\theta(\phi)$. In general $\{ \vec{\phi}(t-n) : n \in \mathbb Z \}$ is not even complete in $V_\theta(\phi)$. For instance, let $\phi(t) = \chi_{[0,1)}(t)$ where $\chi_{[0,1)}(t) := 1$ for $t\in[0,1)$ and $0$ otherwise. Then $\eta_\theta(t)\vec{\phi}(t) = c(\theta)\chi_{[0,1)}(t) \in V_\theta(\phi)$ and $\langle \eta_\theta (t) \vec\phi(t), \vec\phi (t-n) \rangle_{L^2 (\mathbb R)} =\langle c(\theta)\chi_{[0,1)}(t) , e^{ia(t-n)^2}\chi_{[0,1)}(t-n) \rangle_{L^2(\mathbb R)}  = 0$ for $n \in \mathbb Z \backslash\{0\}$. For $\{ \vec{\phi}(t-n) : n \in \mathbb Z \}$ to be complete in $V_\theta(\phi)$, there should be some constant $M$ such that $c(\theta)\chi_{[0,1)}(t) = M e^{iat^2}\chi_{[0,1)}(t) $ for $t \in \mathbb R$, which is true only if $a = 0$, or equivalently $\theta = \frac{\pi}{2} +n\pi$ for $n \in \mathbb Z$.  
\end{remark}

\section{The fractional Fourier duality} \label{FrFourierDuality}

In what follows we always let Assumption \ref{assumption} below hold so that by Proposition \ref{v1prop3.2} and \ref{v0.2prop2.2}, $V_\theta(\phi)$ becomes an RKHS and $\{ c(\theta)\overline{\lambda_\theta(t)} \vec{\phi}(t-n) : n \in \mathbb Z \}$ is a Riesz basis of $V_\theta(\phi)$.

\begin{assumption} \label{assumption}
Let $\Delta := \frac{2\pi}{|b|}={2\pi |\sin\theta|}$ and $\phi(t) \in L^2(\mathbb R)$.
	\begin{itemize}
		\item $\phi(t)$ is well-defined everywhere on $\mathbb R$;
		\item there exist constants $B \geq A >0$ such that $ A \leq |G_{\phi, \theta}(\xi)| \leq B$ a.e. on $[0,\Delta]$;
		\item $C_{\phi,\theta}(t) <\infty$ for $t \in \mathbb R$.
	\end{itemize}
\end{assumption}

We consider the bounded linear operator $\mathcal{J}_\theta$ from $L^2[0,\Delta]$ into $V_\theta (\phi)$, defined by
\begin{equation}\label{v0.5eq3.1}
	(\mathcal{J}_\theta F) (t) := ( \{ \langle F, {K}_\theta (n,\cdot) \rangle_{L^2[0,\Delta]} \}_n \ast_\theta \phi) (t) = \langle F, \lambda_\theta (t) \overline{{c(\theta)} Z_{\phi, \theta}(t, \xi)} \rangle_{L^2[0,\Delta]}. 
\end{equation}

\begin{proposition} [cf. Theorem 1 of \cite{Garcia:2005fz}] \label{v0.8prop4.4}
~
	\begin{enumerate}
		\item[(a)] $\mathcal{J}_\theta$ is an isomorphism from $L^2[0,\Delta]$ onto $V_\theta (\phi)$. 
		\item[(b)] $\mathcal{F}_\theta [\mathcal{J}_\theta F] (\xi) = \overline{\lambda_\theta (\xi)} F(\xi) \hat{\phi}_\theta (\xi)$.
		\item[(c)]  $ \lambda_\theta (t) \mathcal{J}_\theta [F(\cdot)  e^{-ibk\cdot}](t) =\lambda_\theta(t-k) \mathcal{J}_\theta F(t-k)$.  
	\end{enumerate}
\end{proposition}

\begin{proof}
$\mathcal{J}_\theta$ maps the orthonormal basis $\{ K_\theta(n,\xi) : n \in \mathbb Z\}$ of $L^2[0,\Delta]$ into the Riesz basis $\{ c(\theta)\overline{\lambda_\theta(t)} \vec{\phi}(t-n) : n \in \mathbb Z \}$ of $V_\theta (\phi)$, so it is bijective. For (a) it suffices to show that $\mathcal{J}_\theta$ is bounded. This follows since $\| (\mathcal{J}_\theta F) (t)\|^2_{L^2(\mathbb R)} = \| \sum_{n \in \mathbb Z} \vec{d}(n) c(\theta) \overline{\lambda_\theta(t)} \vec{\phi}(t-n)\|^2_{L^2(\mathbb R)} \leq  B \| \vec{\mathbf{d}}\|^2 = B \| {\mathbf{d}}\|^2   = B \| F \|^2_{L^2[0,\Delta]}$ where $\mathbf{d} = \{ d(n) := \langle F, {K}_\theta (n,\cdot) \rangle_{L^2[0,\Delta]} : n \in \mathbb Z \}$ and $B$ denotes the upper bound of the Riesz basis $ \{ c(\theta) \overline{\lambda_\theta(t)} \vec{\phi} (t-n) : n \in \mathbb R \}$ of $V_\theta(\phi)$. 

Since $ \{ c(\theta) \overline{\lambda_\theta(t)} \vec{\phi} (t-n) : n \in \mathbb R \}$ is the Riesz basis of $V_\theta(\phi)$, $\{ \vec{\phi} (t-n) : n \in \mathbb Z \}$ is a Bessel sequence. Thus (b) follows by Lemma \ref{v0.8lemma4.1}. Finally (c) is obtained by \eqref{v0.2eq2}. 
\end{proof}

We are to derive a sampling expansion on $V_\theta(\phi)$ of the following form:
\begin{equation} \label{v0.9aim}
	f(t) = \sum_{n \in \mathbb R} f(\sigma + n) S_n (t), ~f \in V_\theta(\phi)
\end{equation}
where $0\leq \sigma <1$ and $\{ S_n(t):n \in \mathbb Z\}$ is a Riesz basis or a frame of $V_\theta (\phi)$.

Since $V_\theta(\phi)$ is an RKHS, $f(t)$ is well-defined for any $t \in \mathbb R$.
Thus for any $F(\xi) \in L^2[0,\Delta]$ we have from \eqref{v0.2eq2} and \eqref{v0.5eq3.1} that
\begin{eqnarray} \label{v1.7eq14.1}
	f(\sigma+n) = \langle F(\xi), \lambda_\theta (\sigma +n ) \overline{{c(\theta)}Z_{\phi, \theta} (\sigma, \xi)}   e^{-ibn\xi} \rangle_{L^2[0,\Delta]}, ~n \in \mathbb Z
\end{eqnarray} where $f(t) := (\mathcal{J}_\theta F)(t)$. 

\begin{lemma}[cf. Theorem 2 of \cite{Garcia:2005fz}]  \label{v0.3lem3.3} Let $g (\xi) \in L^2 [0,\Delta]$ and $|\rho (n)| = 1$ for $n \in \mathbb Z$.
	\begin{enumerate}
		\item[(a)] $\{ g(\xi)\rho(n) e^{-ibn\xi} : n \in \mathbb Z\}$ is a Bessel sequence with bound $B>0$ if and only if $\Delta |g(\xi)|^2 \leq {B}$ a.e. on $[0,\Delta]$, or equivalently, $g(\xi) \in L^{\infty}[0,\Delta]$.
		\item[(b)] The following are equivalent.
			\begin{enumerate}
				\item[(b1)] $\{ g(\xi)\rho(n) e^{-ibn\xi} : n \in \mathbb Z\}$ is a frame of $L^2[0,\Delta]$ with bounds $(A,B)$.  
				\item[(b2)] $\{ g(\xi)\rho(n) e^{-ibn\xi} : n \in \mathbb Z\}$ is a Riesz basis of $L^2[0,\Delta]$ with bounds $(A,B)$.  
				\item[(b3)] There exist constants $B \geq A >0$ such that ${A} \leq \Delta|g(\xi)|^2 \leq {B}$ a.e. on $[0,\Delta]$.			
			\end{enumerate}
		\item[(c)] $\{ g(\xi)\rho(n) e^{-ibn\xi} : n \in \mathbb Z\}$ is an orthonormal basis of $L^2[0,\Delta]$ if and only if $\Delta|g(\xi)|^2 =1$ a.e. on $[0,\Delta]$.
	\end{enumerate}
\end{lemma}

\begin{proof}
Note that $\{ g(\xi)\rho(n) e^{-ibn\xi} : n \in \mathbb Z\}$ is a Bessel sequence, a frame, a Riesz basis, or an orthonormal basis of $L^2[0,\Delta]$ if and only if $\{ g(\xi)e^{-ibn\xi} : n \in \mathbb Z\}$ is a Bessel sequence, a frame, a Riesz basis, or an orthonormal basis of $L^2[0,\Delta]$, respectively. For any $F(\xi) \in L^2[0,\Delta]$ it follows that $\sum_{n \in \mathbb Z} | \langle F(\xi),g(\xi) e^{-ibn\xi} \rangle_{L^2[0,\Delta]} |^2 = \sum_{n \in \mathbb Z} | \langle F(\xi)\overline{g(\xi)}, e^{-ibn\xi} \rangle_{L^2[0,\Delta]} |^2  = \Delta \| F(\xi)\overline{g(\xi)} \|^2_{L^2[0,\Delta]}$.
We then refer to the proof of Theorem 2 of \cite{Garcia:2005fz} for (a) and (b). (c) can be deduced by essentially the same way as in proofs of (a) and (b). 
\end{proof}

For later use we note that, as an consequence (c) of Lemma \ref{v0.3lem3.3} with $g(\xi) = c(\theta)\lambda_\theta (\xi)$ and $\rho(n) = \lambda_\theta(n)$, $\{ {K}_\theta (n , \xi) : n \in \mathbb Z \}$ is an orthonormal basis of $L^2[0,\Delta]$.

\begin{theorem} \label{v0.8thm4.7}
Let Assumption \ref{assumption} hold.
Let $0 \leq \sigma <1$, and let $\lambda_\theta(\cdot)$ and $C_{\phi,\theta}(t)$ be given by \eqref{v0.7eq1} and \eqref{v0.8eq1.1}, respectively.
The following are equivalent.
\begin{enumerate}
	\item[(a)] There is a frame $\{S_n (t) : n \in \mathbb Z\}$ of $V_\theta(\phi)$ such that 
		\begin{equation} \label{v0.8thmeq4.71}
			f(t) = \sum_{n \in \mathbb Z} f(\sigma+n)S_n (t),~ f \in V_\theta(\phi).
		\end{equation} 
	
	\item[(b)] There is a Riesz basis $\{S_n (t) : n \in \mathbb Z\}$ of $V_\theta(\phi)$ such that 
		\begin{equation} \label{v0.8thmeq4.72}
			f(t) = \sum_{n \in \mathbb Z} f(\sigma+n) S_n (t),~ f \in V_\theta(\phi).
		\end{equation} 
		
	\item[(c)] There are constants $\beta \geq \alpha > 0$ such that 
		\begin{equation} \label{v0.8thmeq4.6}
			\alpha \leq |Z_{\phi, \theta}(\sigma, \xi) | \leq \beta,~\textnormal{ a.e. on } [0,\Delta]
		\end{equation} where $Z_{\phi, \theta}(\cdot, \cdot)$ is the fractional Zak transform of $\phi$, defined by \eqref{v0.9defZAK}.	
\end{enumerate}
In this case \eqref{v0.8thmeq4.71} and \eqref{v0.8thmeq4.72} converge in $L^2(\mathbb R)$ and uniformly on a subset of $\mathbb R$, on which $C_{\phi,\theta}(t)$ is bounded, and
\begin{equation}\label{v0.8thmeq11.1}
	 S_n(t) = \overline{\lambda_\theta(t)} \lambda_\theta (\sigma+n) \lambda_\theta (t-n) S(t-n)
\end{equation}
 where $S(t)$ is in $V_\theta (\phi)$ such that 
 \begin{equation}\label{v0.8eqthm14.1}
 	\hat{S}_\theta (\xi) = \frac{\hat{\phi}_\theta (\xi)}{\lambda_\theta (\xi) Z_{\phi, \theta} (\sigma, \xi)}~\textnormal{ a.e. on } \mathbb R
\end{equation}	
 and $S_n(\sigma+k) = \delta_{n,k}$ for $n,k \in \mathbb Z$.

\end{theorem}

\begin{proof} 
We remind from \eqref{v1.7eq14.1} that for any $F(\xi) \in L^2 [0,\Delta]$,
\begin{eqnarray*}
	f(\sigma+n) = \langle F(\xi),  \overline{{c(\theta)}Z_{\phi, \theta} (\sigma, \xi)}  \lambda_\theta (\sigma +n ) e^{-ibn\xi} \rangle_{L^2[0,\Delta]}, ~n \in \mathbb Z
\end{eqnarray*} where $f(t) := (\mathcal{J}_\theta F)(t)$. 
Assume (c). 
By (b) of Lemma \ref{v0.3lem3.3} with $g(\xi)= \overline{{c(\theta)}Z_{\phi,\theta}(\sigma, \xi)}$ and $\rho(n)= \lambda_\theta(\sigma+n)$, $ \{\overline{{c(\theta)}Z_{\phi, \theta} (\sigma, \xi)}  \lambda_\theta (\sigma +n ) e^{-ibn\xi} : n \in \mathbb Z\}$ is a frame (or equivalently, a Riesz basis) of $L^2[0, \Delta]$. Then there exists its dual in $L^2[0,\Delta]$ and it is easy to see that the dual  is of the form $ \{h_\theta(\xi) \lambda_\theta (\sigma +n ) e^{-ibn\xi} : n \in \mathbb Z\}$ for some $h_\theta(\xi) \in L^2 [0,\Delta]$. Thus we obtain a frame (or equivalently, a Riesz basis) expansion of $F(\xi)$ in $ L^2 [0,\Delta]$ as
\begin{equation}\label{v0.8eq13}
	F(\xi) = \sum_{n \in \mathbb Z}f(\sigma+n) h_\theta(\xi) \lambda_\theta (\sigma +n ) e^{-ibn\xi}.
\end{equation} Via $\mathcal{J}_\theta$, \eqref{v0.8eq13} is equivalent to $f(t) = \sum_{n \in \mathbb Z} f(\sigma+n) \mathcal{J}_\theta [h_\theta(\xi) \lambda_\theta (\sigma +n ) e^{-ibn\xi}](t)$ for $f(t) \in V_\theta (\phi)$, which proves (a) (or equivalently (b)) by setting $S_n (t) =  \mathcal{J}_\theta [h_\theta(\xi) \lambda_\theta (\sigma +n ) e^{-ibn\xi}](t)$. In this case, by (c) of Proposition \ref{v0.8prop4.4}, setting $S(t) := \mathcal{J}_\theta [h_\theta ](t) \in V_\theta(\phi)$, we obtain \eqref{v0.8thmeq11.1} . 
Now assume (a). Applying $\mathcal{J}^{-1}_\theta$ on \eqref{v0.8thmeq4.71}, we have \eqref{v0.8eq13} where $F(\xi) : = \mathcal{J}^{-1}_\theta[f](\xi)$. Then (c) follows from (b) of Lemma \ref{v0.3lem3.3} and so does the equivalence between (a) and (b). The convergence mode follows since $V_\theta (\phi)$ is an RKHS in which any $L^2$-convergent sequence also converges uniformly where $\| q_\theta(t,\cdot)\|_{L^2(\mathbb R)} = q_\theta(t,t) = C_{\phi,\theta}(t) $ is bounded. Here, $q_\theta(\cdot, \cdot)$ denotes the reproducing kernel of $V_\theta(\phi)$. Finally we have \eqref{v0.8eqthm14.1} by substituting $\phi(t)$ into $f(t)$ on \eqref{v0.8thmeq4.71}, followed by taking $\mathcal{F}_\theta$ on both sides. 
\end{proof}

${Z_{\phi, \theta} (t, \xi)}$, as a function of $\xi$, is not $\Delta$-periodic. So $\alpha$ and $\beta$ in \eqref{v0.8thmeq4.6} vary depending on an interval of length $\Delta$. For instance, one can take an interval as $[-\Delta/2,\Delta/2]$ while we take as $[0,\Delta]$ in \eqref{v0.5eq3.1}. However it is interesting to note that the synthesis function $S_n (t)$, given by \eqref{v0.8thmeq11.1}, is independent of the interval. Based on this observation, we have:
\begin{theorem} \label{v1.2thm4.6}
Let the assumption and the notation be the same as in Theorem \ref{v0.8thm4.7}. Then the equivalent statements of Theorem \ref{v0.8thm4.7} are also equivalent to 
\begin{enumerate}
	\item[(d)] There are constants $\beta \geq \alpha > 0$ such that 
		\begin{equation*} 
			\alpha \leq |Z_{\phi, \theta}(\sigma, \xi) | \leq \beta,~\textnormal{ a.e. on } E
		\end{equation*} where $E$ is an interval of length $\Delta$ in $\mathbb R$, and $Z_{\phi, \theta}(\cdot, \cdot)$ is the fractional Zak transform of $\phi$, defined by \eqref{v0.9defZAK}.	
\end{enumerate}
\end{theorem}

\begin{proof}
We need to show that $S_n (t)$ (or equivalently, $S(t)$) is given independently of the choice of an interval of length $\Delta$.
Since the denominator $\lambda_\theta (\xi) Z_{\phi, \theta} (\sigma, \xi)={\sum_{n \in \mathbb Z} \vec\phi(t-n)e^{ibn\xi}}$ of \eqref{v0.8eqthm14.1} is $\Delta$-periodic, the conclusion follows.
\end{proof}

Using an orthonormal basis expansion of $V_\theta(\phi)$, we also have:
\begin{theorem} \label{v0.9thm4.8}
Let Assumption \ref{assumption} hold.
Let $0 \leq \sigma <1$, and $C_{\phi,\theta}(t)$ be given by \eqref{v0.8eq1.1}. The following are equivalent.
\begin{enumerate}
	\item[(a)] There is an orthonormal basis $\{S_n (t) : n \in \mathbb Z\}$ of $V_\theta(\phi)$ such that 
		\begin{equation*}
			f(t)= \sum_{n \in \mathbb Z} f(\sigma+n) S_n(t),~ f \in V_\theta(\phi)
		\end{equation*} which converges in $L^2(\mathbb R)$ and uniformly on a subset of $\mathbb R$, on which $C_{\phi,\theta}(t)$ is bounded.
	\item[(b)] $ |Z_{\phi, \theta}(\sigma, \xi) |^2 = 1$ a.e. on $E$, where $E$ is an interval of length $\Delta$ in $\mathbb R$ and $Z_{\phi,\theta}(\cdot, \cdot)$ is the fractional Zak transform of $\phi$, defined by \eqref{v0.9defZAK}.
\end{enumerate}
In this case, $ S_n(t)$ is given by \eqref{v0.8thmeq11.1} and $S_n(\sigma+k) = \delta_{n,k}$ for $n,k \in \mathbb Z$.
\end{theorem}

\begin{proof}
We borrow the notation from the proof of Theorem \ref{v0.8thm4.7}. 
Then it is enough to show that $ \{\overline{{c(\theta)}Z_{\phi, \theta} (\sigma, \xi)}  \lambda_\theta (\sigma +n ) e^{-ibn\xi} : n \in \mathbb Z\}$ is an orthonormal basis of $L^2[0, \Delta]$ if and only if $ |Z_{\phi, \theta}(\sigma, \xi) |^2 = 1$ a.e. on $[0,\Delta]$. This follows by (c) of Lemma \ref{v0.3lem3.3}.
\end{proof}

\begin{remark} \label{v1.2remark4.8}
Theorem \ref{v0.8thm4.7} and \ref{v0.9thm4.8} are also derived from Theorem 3.2. of \cite{Kim:2008ik}, combined with Proposition 3.1. of \cite{Kim:2008ik}, since $V_\theta(\phi) $ is an isomorphic image of $V(\vec{\phi})$. Statements and formulas in Theorem 3.2. of \cite{Kim:2008ik} can be translated into the corresponding ones by means of \eqref{v1eq2.01} and \eqref{v1eq2.02}-\eqref{v1eq2.03}. For instance, the equation (3.12) of \cite{Kim:2008ik} agrees with \eqref{v0.8eqthm14.1}. 
\end{remark}

From Remark \ref{v1.2remark4.8}, it is not difficult to see that the following are also equivalent to the statements of Theorem \ref{v0.8thm4.7} and \ref{v1.2thm4.6}.
	\begin{enumerate}
		\item[(e)] There are constants $\beta \geq \alpha >0$ such that 
		\begin{equation*}
			\alpha \| f\|^2_{L^2(\mathbb R)} \leq \sum_{n \in \mathbb Z} |f(\sigma+n)|^2 \leq \beta \| f\|^2_{L^2(\mathbb R)},~f \in V_\theta(\phi).
		\end{equation*}
		
		\item[(f)] $\{ q_\theta (\cdot,\sigma+n) : n \in \mathbb Z \}$ is a Riesz basis of $V_\theta(\phi)$ where $q_\theta (\cdot, \cdot)$ is the reproducing kernel of $V_\theta(\phi)$.
	\end{enumerate}

%
%
%

\begin{example}
Let $\phi(t) =\overline{\lambda_\theta(t)}\textnormal{sinc}(t)$ and $\sigma = 0$. Then $\sup_{\mathbb R} C_{\phi, \theta} (t) < \infty$, and $ Z_{\phi, \theta}(0,\xi) = \sum_{n \in \mathbb Z} \textnormal{sinc}(-n)\lambda_\theta(n) \overline{K_\theta (n,\xi)} = \overline{K_\theta (0,\xi)}  = \overline{c(\theta)}e^{-ia\xi^2}$ so that $\Delta |Z_{\phi, \theta}(0,\xi)|^2 = 1$ on $[0,\Delta]$. By Theorem \ref{v0.9thm4.8} we have {an orthonormal basis} expansion:
\begin{equation} \label{v1.3exeq4.9}
 f(t) =  e^{-iat^2} \sum_{n \in \mathbb Z} f(n) e^{ian^2} \textnormal{sinc} (t-n),~ f \in V_\theta (\phi)
\end{equation} 
which converges in $L^2(\mathbb R)$ and uniformly on $\mathbb R$. \eqref{v1.3exeq4.9} is known as the WSK sampling expansion for the FrFT domain, first derived in \cite{Xia:1996fe,Zayed:1996hg}. 
\end{example}

\bigskip
\noindent {\bf Acknowledgement} \\
The author thanks Prof. Ahmed I. Zayed for useful comments and suggestions during the preparation of the manuscript.
This work is partially supported by EM-BEAM program funded by the European Commision.


\end{document}